\documentclass[12pt, reqno]{amsart}  
  \usepackage{amsmath,amssymb,latexsym,footnote,dsfont}
 \usepackage{graphicx,pstricks}

\usepackage{fullpage}
\usepackage{mathrsfs}
\usepackage{amsfonts}
\usepackage{amsthm}
\usepackage{amsfonts}
\usepackage[toc,page]{appendix}
\usepackage{ mathrsfs }
\usepackage{amsmath}
\usepackage[colorlinks=true, pdfstartview=FitV, linkcolor=blue, citecolor=blue, urlcolor=blue]{hyperref}
\makeatletter  
\def\@@insvline#1#2{{\setbox0\hbox{\m@th$#1\mathrm I$}  
  \rlap{\m@th$#1 \mkern 5mu  
  \vrule height.95\ht0 depth-.005\ht0 width.09\ht0 $}  
  {\mathrm #2} }}

\def\Q{\mathpalette\@@insvline{Q}}

\makeatother  
  \newtheorem{defi}{Definition}
  \newcommand{\bd}{\begin{defi}} 
  \newcommand{\ed}{\end{defi}}

  \newtheorem{lemm}[defi]{Lemma}  
  \newcommand{\bl}{\begin{lemm}}
  \newcommand{\el}{\end{lemm}} 

  \newtheorem{theo}[defi]{Theorem}
  \newcommand{\bt}{\begin{theo}}
  \newcommand{\et}{\end{theo}}

  \newtheorem{cor}[defi]{Corollary}
  \newcommand{\bc}{\begin{cor}}
  \newcommand{\ec}{\end{cor}}

  \newtheorem{pro}[defi]{Proposition}
  \newcommand{\bp}{\begin{pro}}
  \newcommand{\ep}{\end{pro}}
  
  \newcommand{\bydef}{\overset{def}{=}}

  \makeatletter
  \def\proof{\@ifnextchar[\opargproof{\opargproof[\bf Proof \hfil\\ ]}}
  \def\opargproof[#1]{\par\noindent {\bf #1 }}
  
  \makeatother
\DeclareFontFamily{OT1}{nice}{}
\DeclareFontShape{OT1}{nice}{m}{n}{<5> <6> <7> <8> <9> <10>
<12><10.95><14.4><17.28><20.74><24.88>callig15}{}
\DeclareFontFamily{U}{nice}{}
\DeclareFontShape{U}{nice}{m}{n}{<5> <6> <7> <8> <9> <10>
<12><10.95><14.4><17.28><20.74><24.88>callig15}{}
\DeclareSymbolFont{calligra}{U}{nice}{m}{n}
\DeclareSymbolFontAlphabet{\nice}{calligra}
\DeclareFontFamily{OT1}{cmdh}{}
\DeclareFontShape{OT1}{cmdh}{m}{n}{<10>cmdunh10}{}
\input umsa.fd
\input umsb.fd
\input ueuex.fd
\input ueuf.fd
\input ueur.fd
\input ueus.fd

\def\epsilon{\varepsilon}  
\def\phi{\varphi}

\newtheorem{thm}{Theorem}
\newtheorem{rmq}{Remark}
\newtheorem{lemma}{Lemma}

\newtheorem{defin}{Definition}

\newcommand{\norm}[1]{\left\Vert #1\right\Vert}
\newcommand\blfootnote[1]{%
  \begingroup
  \renewcommand\thefootnote{}\footnote{#1}%
  \addtocounter{footnote}{-1}%
  \endgroup}

\date{\today}

\author[H. Houamed]{Haroune Houamed}
\address{CNRS, LJAD, Université Co\^te d'Azur\\ Département de Mathématiques\\ Nice,  France}
\email{haroune.houamed@univ-cotedazur.fr}
\thanks{This Work has been done when the author was a PhD student in the University of Nice-C\^ote d'Azur-France, under the supervision of F.Planchon and P.Dreyfuss. In particular, the author would like to thank his supervisors for the accomplished work.}
\title{About some possible blow-up conditions for the 3-D Navier-Stokes equations}
\begin{document}

\begin{abstract}  In this paper, we study some conditions related to the question of the possible blow-up of regular solutions to the 3D Navier-Stokes equations. In particular, up to a modification in a proof of a very recent result from \cite{Isab}, we prove that if one component of the velocity remains small enough in a sub-space of $\dot{H}^{\frac{1}{2}}$ "almost" scaling invariant, then the 3D Navier Stokes is globally wellposed. In a second time, we investigate the same question under some conditions on one component of the vorticity and unidirectional derivative of one component of the velocity in some critical Besov spaces of the form $L^p_T(\dot{B}_{2,\infty}^{\alpha, \frac{2}{p}-\frac{1}{2}-\alpha})$ or $L^p_T(\dot{B}_{q,\infty}^{ \frac{2}{p}+\frac{3}{q}-2})$.\end{abstract}  
\maketitle
\blfootnote{{\it keywords:}
Incompressible Navier-Stokes Equations, Anisotropic Littlewood-Paley Theory, Blow-up criteria.}  

\section{Introduction}
In this work we are interested in the study of the possible blow-up for regular solutions to the 3D incompressible Navier stokes equations
\begin{center}
$(NS)\left\{\begin{array}{l}
 \partial_t u + u \cdot \nabla u - \Delta u + \nabla P=0, \; \; \; (t,x) \in \mathbb{R}^+ \times \mathbb{R}^3\\
  div\, u = 0 \\
  u_{|t=0}= u_0,
\end{array}\right.$
\end{center}
where the unknowns of the equations $u=(u^1,u^2,u^3)$, $P$ are respectively, the velocity and the pressure of the fluid. 
We recall that the set of the solutions to $(NS)$ is invariant under the transformation:
\begin{center}
$\begin{array}{cc}
u_{0,\lambda}(x) \overset{def}{=} \lambda u_0(\lambda x), & u_{\lambda}(t,x)  \overset{def}{=} \lambda u_{\lambda}(\lambda^2 t, \lambda x) .
\end{array}$
\end{center}
That is if $u(t,x)$ is a solution to $(NS)$ on $[0,T]\times \mathbb{R}^3$ associated to the initial data $u_0$ , then, for all $\lambda>0$, $ u_{\lambda}(t,x) $ is a solution to $(NS)$ on $[0,\lambda^{-2}T]\times \mathbb{R}^3$ associated to the initial data $ u_{0,\lambda}$.\\
It is well known that system $(NS)$ has at least one global weak solution with finite energy
\begin{equation}\label{L2-energy}
\norm {u(t)}_{L^2}^2 + 2\int_0^t \norm {\nabla u(t')}_{L^2}^2 dt' \leq \norm {u_0}_{L^2}^2.
\end{equation}
This result was proved first by J.Leray in \cite{Leray1}. In dimension three, uniqueness for such solutions stands to be an open problem. J.Leray proved also in his famous paper \cite{Leray1} that, for more regular initial data, namely for $u_0\in H^1(\mathbb{R}^3)$, $(NS)$ has a unique local smooth solution. That is, there exists $T^*>0$ and a unique maximal solution $u$ in $ L^{\infty}_{T^*}(H^1(\mathbb{R}^3)) \cap L^2_{T^*}(H^2(\mathbb{R}^3))$. The question of the behavior of this solution after $T^*$ remains to be also an open problem.\\
In order to give a ``formally" large picture, let us define the set
\begin{equation}\label{chi_T}
\chi_T \overset{def}{=} \bigg( (L^{\infty}_T L^2 \cap L^2_T H^1) \cdot (L^{\infty}_T L^2 \cap L^2_T H^1) \bigg)',
\end{equation}
where
$$(L^{\infty}_T L^2 \cap L^2_T H^1) \cdot (L^{\infty}_T L^2 \cap L^2_T H^1) \bydef \big\{ uv : \; (u,v)\in (L^{\infty}_T L^2 \cap L^2_T H^1) \times (L^{\infty}_T L^2 \cap L^2_T H^1)\big \}.$$
Multiplying $(NS)$ by $-\Delta u$, and integrating by parts yield
\begin{equation*}
\frac{1}{2}\frac{d}{dt} \norm {\nabla u}_{L^2}^2 + \norm {\nabla u}_{H^1}^2 = - \int_{\mathbb{R}^d} \big(\nabla u \cdot \nabla u \big| \nabla u \big).
\end{equation*}
If we suppose that $\nabla u$ is already bounded in $G_T$ some sub-space of $\chi_T $, then one may prove that $\nabla u$ is bounded in  $L^{\infty}_T L^2 \cap L^2_T H^1$. This is the case in dimension two where we get, for free, by the $L^2$-energy estimate \eqref{L2-energy} a uniform bound of $\nabla u $ in $L^2_TL^2 \subset \big((L^4_T L^4) \cdot L^4_T L^4) \big)' \subset \chi_T $.\\
In the case of dimension three, several works have been done in this direction, establishing the global wellposedness of $(NS)$ under assumptions of the type $\nabla u\in G_T$. We can set as an example of these results the well known Prodi-Serrin type criterion, saying that, if $u\in L^p([0,T], L^q(\mathbb{R}^3) )$, with $\frac{2}{p}+\frac{3}{q}=1$ and $q\in ]3,\infty]$, then $(NS)$ is globally wellposed. The limit case where $q=3$ was proved recently by L. Escauriaza, G. Seregin and V. Sver\`ak in \cite{Sverak} proving that: if $T^*\bydef T^*(u_0)$ denotes the life span of a regular solution $u$ associated to the initial data $u_0$ then
\begin{equation}\label{L-3-criterion}
T^* < \infty \implies \limsup_{t\rightarrow T^*} \norm {u(t)}_{L^3(\mathbb{R}^3)} =\infty.
\end{equation} 
This was extended to the full limit in time in $\dot{H}^{\frac{1}{2}}(\mathbb{R}^3)$ by G. Seregin in \cite{Seregin}. And more recently in \cite{Gallagher-Koch-Planchon}, the $L^3$-norm in \eqref{L-3-criterion} was extended to the critical Besov spaces $\dot{B}^{-1+\frac{3}{p}}_{p,q}$, for any $3<p,q<\infty$.\\

On another hand, one may notice that the divergence free condition can provide us another type of conditions for the global regularity (let us say anisotropic ones) under conditions on some components of the velocity or its gradient. Several works have been done in this direction, one may see for instance \cite{Adhikari,Penel1,Penel2,Cao1,Cao2,Planchon1,He1,Ziane1,Panel3,Pokorny1,Pokorny2,Yao1} for examples in some scaling invariant spaces or not of Serrin-type regularity criterion, or equivalently proving that, if $T^*$ is finite then
\begin{align*}
\int_{0}^{T^*} \norm {u^3(t,.)}_{L^q}^pdt=\infty \;\;\; \text{ or } \;\;\; \int_{0}^{T^*} \norm {\partial_j u^3(t,.)}_{L^q}^pdt=\infty.
\end{align*}
The first result in a scaling invariant space under only one component of the velocity has been proved by J.-Y Chemin and P.Zhang in \cite{Chemin2} for  $p\in ]4,6[$ and a little bit later by the same authors together with Z.Zhang in \cite{Chemin4} for $p \in ]4,\infty[$. The case $p=2$ has been treated very recently by J.-Y Chemin, I.Gallagher and P.Zhang in \cite{Isab}.  As mentioned in \cite{Isab} such a result in the case of $p= \infty$, assuming it is true, seems to be out of reach for the time being.\\

 However the authors in \cite{Isab} proved some results for $p=\infty$. Mainly they proved that if there is a blow-up at some time $T^*>0$, then it is not possible for one component of the velocity to tend to 0 too fast. More precisely they proved the following blow-up condition
\begin{equation*}  
\forall \sigma\in \mathbb{S}^2, \forall t< T^*, \; \; \; \sup_{t'\in[t,T^*[}\norm {u(t')\cdot\sigma }_{\dot{H}^{\frac{1}{2}}} \geq c_0 log^{-\frac{1}{2}} \bigg( e+ \frac{\norm {u(t)}_{L^2}^4}{T^*-t} \bigg).
\end{equation*}
The last result proved in their paper needs reinforcing slightly the $\dot{H}^{\frac{1}{2}}$ norm in some directions. Mainly, without loss of generality, their result can be stated as the following
\begin{thm}\label{log_hcase}
There exists a positive constant $c_0$ such that if $u$ is a maximal solution of $(NS)$ in $C([0,T^*[,H^1)$, then for all positive real number $E$ we have:
$$ \;\; T^* < \infty \implies \limsup_{t\rightarrow T^*} \norm {u^3(t)}_{\dot{H}^{\frac{1}{2}}_{log_h,E}}\geq c_0,$$
where
$$\norm a_{\dot{H}^{\frac{1}{2}}_{log_h,E}}^2\overset{def}{=} \int_{\mathbb{R}^3} |\xi| log(E|\xi_h|+e) |\hat{a}(\xi)|^2 d\xi < \infty . $$
\end{thm}
 Motivated by this result, we aim to show that, up to a small modification in the proof of Theorem \ref{log_hcase}, we can obtain the same blow-up condition in the case $p=\infty$, by slightly reinforcing the $\dot{H}^{\frac{1}{2}}$ norm in the vertical direction instead of the horizontal one. More precisely, we define
\begin{defin}
Let $E$ be a positive real number. We define $\dot{H}^{\frac{1}{2}}_{log_v,E}$ to be the sub space of $\dot{H}^{\frac{1}{2}}(\mathbb{R}^3)$ such that:
$$a\in \dot{H}^{\frac{1}{2}}_{log_{v},E} \iff \norm a_{\dot{H}^{\frac{1}{2}}_{log_v,E}}^2\overset{def}{=} \int_{\mathbb{R}^3} |\xi| log(E|\xi_v|+e) |\hat{a}(\xi)|^2 d\xi < \infty . $$
\end{defin}
 We will prove
\begin{thm}\label{log_vcase}
There exists a positive constant $c_0$ such that if $u$ is a maximal solution of $(NS)$ in $C([0,T^*[,H^1)$, then for all positive real number $E$ we have
$$ \;\; T^* < \infty \implies \limsup_{t\rightarrow T^*} \norm {u^3(t)}_{\dot{H}^{\frac{1}{2}}_{log_v,E}}\geq c_0.$$
\end{thm}
\begin{rmq}
The blow-up condition stated in Theorem \ref{log_vcase} above can be generalized to the following one
$$\forall \sigma\in \mathbb{S}^2, \;\; T^* < \infty \implies \limsup_{t\rightarrow T^*} \norm {\sigma \cdot u(t)}_{\dot{H}^{\frac{1}{2}}_{log_ {\tilde{\sigma}},E}}\geq c_0,$$
where
$$ \norm a_{\dot{H}^{\frac{1}{2}}_{log_{\tilde{\sigma},E}}}^2 \overset{def}{=} \int_{\mathbb{R}^3} |\xi| log(E|\xi_{\tilde{\sigma}}|+e) |\hat{a}(\xi)|^2 d\xi < \infty , \; \; \; \text{with } \xi_{\tilde{\sigma}}\overset{def}{=} (\xi\cdot\sigma)\sigma. $$
\end{rmq}
\begin{rmq}
We can show also that, if there is a blow-up at a finite time $T^*$, then
$$\limsup_{t\rightarrow T^*} \norm {   u^3(t)}_{\dot{H}^{\frac{1}{2}}_{log_E}}\geq c_0,$$
where we set
$$\norm a_{\dot{H}^{\frac{1}{2}}_{log_E}}^2\overset{def}{=} \int_{\mathbb{R}^3} |\xi| log(E\min\{|\xi_h|,|\xi_v|\}+e) |\hat{a}(\xi)|^2 d\xi < \infty . $$
This can be done by following the same ideas in the proof of Theorems \ref{log_hcase} and \ref{log_vcase}, together with the fact that
$$\norm{u^3}_{\dot{H}^{\frac{1}{2}}_{log_E}} ^2 = \norm{u^3_h}_{\dot{H}^{\frac{1}{2}}_{log_{v ,E}}} ^2+\norm{u^3_v}_{\dot{H}^{\frac{1}{2}}_{log _{h,E}}} ^2,$$
with
$$u^3_h \bydef \mathcal{F}^{-1}\big(\mathrm{1}_{|\xi_v|\leq |\xi_h|} \widehat{u^3}(\xi)\big), \quad u^3_v \bydef \mathcal{F}^{-1}\big(\mathrm{1}_{|\xi_h|<|\xi_v|} \widehat{u^3}(\xi)\big) . $$
 We refer the reader to \cite{Haroune-thesis} for more details.
\end{rmq}
 The other two results that we will prove in this paper can be seen as some blow-up criteria under scaling invariant conditions on one component of the velocity and one component of the vorticity, whether in some anisotropic Besov spaces of the form $L^p\big((B_{2,\infty}^{\alpha})_h(B_{2,\infty}^{s_p-\alpha})_v \big)$, for $\alpha\in [0,s_p]$, or $L^p(\mathcal{B}_{q,p})$, where
\begin{align}\label{notation1}
 s_p \overset{def}{=} \frac{2}{p}-\frac{1}{2}, \;\; \text{and} \;\; \mathcal{B}_{q,p} \bydef \dot{B}^{\frac{3}{q}+\frac{2}{p}-2}_{q,\infty}.
\end{align}
We will prove
\begin{thm}\label{anisotropicbesovthm}
Let $u$ be a maximal solution of $(NS)$ in $C([0,T^*[;H^1)$. If $T^*< \infty$, then
\begin{equation*}
\forall p,m \in [2,4],\; \forall \alpha \in \bigg[0, \frac{2}{p}- \frac{1}{2}\bigg], \; \forall \beta \in \bigg[0, \frac{2}{m}- \frac{1}{2}\bigg], \text{ we have:}
\end{equation*}
\begin{equation*}
\int_0^{T^*} \norm {\partial_3u^3(t')}_{\dot{B}^{\alpha, s_p- \alpha}_{2,\infty}}^p dt' +
\int_0^{T^*} \norm {\omega^3(t')}_{\dot{B}^{\beta, s_m- \beta}_{2,\infty}}^m dt' = \infty.
\end{equation*}
\end{thm}
\begin{thm}\label{negativebesovthm}
Let $u$ be a maximal solution of $(NS)$ in $C([0,T^*[;H^1)$. If $T^*< \infty$, then for all $q_1,q_2 \in [3,\infty[$, for all $ p_1,p_2$ satisfying 
\begin{equation}\label{assumption1}
\frac{3}{q_i}+ \frac{2}{p_i} \in ]1,2[ ,\;\; i\in \{ 1,2 \},
\end{equation}
we have
\begin{equation*}
\int_0^{T^*} \norm {\partial_3u^3(t')}_{\mathcal{B}_{q_1,p_1}}^{p_1} dt' +
\int_0^{T^*} \norm {\omega^3(t')}_{\mathcal{B}_{q_2,p_2}}^{p_2} dt' = \infty.
\end{equation*}
\end{thm}  
A bunch of remarks and comments are listed below:
\begin{rmq}
All the spaces stated in Theorem \ref{anisotropicbesovthm} and Theorem \ref{negativebesovthm} above are scaling invariant spaces under the natural 3-D Navier Stokes scaling.
\end{rmq}
\begin{rmq}
The regularity of the spaces stated in the blow-up conditions in Theorem \ref{negativebesovthm} is negative, more precisely under assumption \eqref{assumption1}, $\frac{3}{q_i}+\frac{2}{p_i}-2\in]-1,0[$. Moreover, the integrability asked for in the associated Besov spaces is always higher than $3$, which make these spaces larger than $L^p_T\big( \dot{H}^{\frac{2}{p}-\frac{1}{2}}\big)$.
\end{rmq}
\begin{rmq}
Taking in mind the embedding $L^p_T\big( \dot{H}^{\frac{2}{p}-\frac{1}{2}}\big) \hookrightarrow L^p_T\big(\dot{B}^{\alpha, \frac{2}{p}-\frac{1}{2}- \alpha}_{2,\infty}\big)$, for all $\alpha\in [0,\frac{2}{p}-\frac{1}{2}]$ (see lemma \ref{embedding-anisotropic}), it is obvious then that the blow-up conditions stated in Theorem \ref{anisotropicbesovthm} imply the ones in $L^p_T\big( \dot{H}^{\frac{2}{p}-\frac{1}{2}}\big)$.
\end{rmq}
\begin{rmq}
In the case $p=4$ (resp. $m=4$) in Theorem \ref{anisotropicbesovthm}, $\alpha$ (resp. $\beta $) is necessary zero, this means that the anisotropic space above is nothing but $L^4_T(\dot{B}^0_{2,\infty})$, which is still larger than $L^4_T(L^2)$. The proof in this case can be done without any use of anisotropic techniques.
\end{rmq}
 The structure of the paper is the following: In section 2, we reduce the proof of the Theorems to the proofs of three lemmas. In Section 3, we should present the proofs of these three lemmas, where we will use some results which will be recalled/proved in the Appendix, together with the definition and the properties of the functional spaces used in this work.\\

\noindent \underline{\textit{\textbf{Notations}}}: In the sequel, we will be using the following notations:\\

 If $A$ and $B$ are two real quantities, the notation $A\lesssim B$ means $A\leq CB$ for some universal constant $C$ which is independent on varying parameters of the problem.\\
$(c_q)_{q\in\mathbb{Z}}$ (resp. $(d_q)_{q\in\mathbb{Z}}$) will be a sequence satisfying $\displaystyle\sum_{q\in \mathbb{Z}} c_q^2 \leq 1$ (resp. $\displaystyle\sum_{q\in \mathbb{Z}} d_q \leq 1$), which is allowed to differ from a line to another one.\\
We also set-up the following notations
\begin{center}
$\begin{array}{cc}
L^r_T(L^p_hL^q_v) \bydef L^r((0,T); L^p((\mathbb{R}^2_h); L^q(\mathbb{R}_v))),  & \dot{H}^s_h(\dot{H}^t_v)  \bydef  \dot{H}^{s,t}(\mathbb{R}^3), \\
\norm {\cdot}_{\dot{H}^s_h(\dot{H}^t_v)} \bydef \norm {\cdot}_{\dot{H}^{s,t}(\mathbb{R}^3)},&\norm {\cdot}_{\dot{B}^{s}_{p,q} } \bydef \norm {\cdot}_{\dot{B}^{s}_{p,q} (\mathbb{R}^3)}.
\end{array}$

\end{center}
\section{Proof of the Theorems}
Let $(i,\ell)\in \{1,2 \}^2$, we denote:
\begin{equation*}
J_{i\ell}(u,u^3) \overset{\text{def}}{=} \int_{\mathbb{R}^3} \partial_i u^3 \partial_3 u^{\ell} \partial_i u^{\ell}.
\end{equation*} 
The proof of Theorem \ref{log_vcase} is then based on the following lemma
\begin{lemma}\label{hardtermlemma}
There exists $C>0$ such that, for any $E>0$, we have:
\begin{equation*}
\big| J_{i\ell}(u,u^3)\big| \leq \bigg(\frac{1}{10}+ C \norm {u^3}_{\dot{H}^{\frac{1}{2}}_{log_v,E}} \bigg) \norm {\nabla_h u}_{\dot{H}^1}^2 	+C \norm {u^3}_{\dot{H}^{\frac{1}{2}}} ^2\frac{\norm{\partial_3 u^h}_{L^2}^2}{E^2}.
\end{equation*}
\end{lemma}
 While, the proofs of Theorem \ref{anisotropicbesovthm} and Theorem \ref{negativebesovthm} are essentially based on the following ones
\begin{lemma}\label{lemmafortheorems3}
For all $p\in ]2,4[$, for all $\alpha\in ]0,s_p[$, where $s_p=\frac{2}{p}-\frac{1}{2}$, we have:
\begin{equation*}
\big|  \big(  fg | g \big)_{L^2} \big| \leq \frac{1}{10} \norm {g}_{\dot{H}^1(\mathbb{R}^3)}^2+ C \norm {f}_{ \dot{B}^{\alpha,s_p- \alpha}_{2,\infty}}^p \norm {g}_{L^2(\mathbb{R}^3)}^2.
\end{equation*}
\end{lemma}
\begin{lemma}\label{lemmafortheorems2}
For any $p,q\in [1,\infty]$ satisfying $\frac{3}{q}+ \frac{2}{p} \in ]1,2[$ we have
\begin{equation*}
\big| \big(  fg | g \big)_{L^2} \big| \leq \frac{1}{10} \norm {g}_{\dot{H}^1(\mathbb{R}^3)}^2+ C \norm {f}_{\mathcal{B}_{q,p}}^p \norm {g}_{L^2(\mathbb{R}^3)}^2.
\end{equation*}
\end{lemma}
 As mentioned above, let us assume that lemmas \ref{hardtermlemma}, \ref{lemmafortheorems3} and  \ref{lemmafortheorems2} hold true, which we will prove in the next section, and let us prove Theorems \ref{log_vcase}, \ref{anisotropicbesovthm} and \ref{negativebesovthm}.
\subsection{\textit{Proof of Theorem \ref{log_vcase}}}
\begin{proof}[]
 Following the idea of \cite{Isab} we begin by establishing a bound of $\nabla_h u$ in $L^{\infty}_T(L^2) \cap L^2_T(\dot{H}^1)$, then we use this estimate to prove a bound of $\partial_3 u$ in $L^{\infty}_T(L^2) \cap L^2_T(\dot{H}^1)$. To do so we multiply $(NS)$ by $-\Delta_h u$, usual calculation leads then to:
\begin{equation}\label{innerproduct}
\frac{1}{2}\frac{d}{dt} \norm {\nabla_h u }_{L^2}^2 + \norm {\nabla_h u}_{\dot{H}^1}^2 = \sum_{i=1}^4\mathcal{E}_i(u) \text{    with: }
\end{equation} 
\begin{align*}
\mathcal{E}_1(u) &\overset{\text{def}}{=} -\sum_{i=1}^2  \big(\partial_i u^h \cdot \nabla_h u^h \big| \partial_i u^h \big)_{L^2}, \\
\mathcal{E}_2(u) &\overset{\text{def}}{=} -\sum_{i=1}^2\big(\partial_i u^h \cdot \nabla_h u^3 \big| \partial_i u^3 \big)_{L^2},\\
\mathcal{E}_3(u) &\overset{\text{def}}{=} -\sum_{i=1}^2 \big(\partial_i u^3 \partial_3 u^h \big| \partial_i u^h \big)_{L^2},\\
\mathcal{E}_4(u) &\overset{\text{def}}{=} -\sum_{i=1}^2  \big(\partial_i u^3 \partial_3 u^h \big| \partial_i u^h \big)_{L^2}.
\end{align*}
A direct computation shows that $\mathcal{E}_1(u),\mathcal{E}_2(u)$ and $\mathcal{E}_4(u)$ can be expressed as a sum of terms of the form
\begin{equation*}
I(u)\overset{\text{def}}{=} \int_{\mathbb{R}^3} \partial_i u^3 \partial_j u^k \partial_{\ell} u^m,
\end{equation*}
where: $(j,\ell)\in \{1,2 \}^2$ and $(i,k,m)\in\{ 1,2,3\}^3$.\\
Next, by duality, product rules and then interpolation, for any $p\in ]1,+\infty]$, one may easily show that\footnote{Notice that $I(u)$ provides a global bound if $u^3 \in L^p(\dot{H}^{\frac{2}{p}+\frac{1}{2}})$ for some $p\in ]1,\infty]$. It is in fact the term $\mathcal{E}_3(u)$ which poses a problem, and this is why this method doesn't give a complete answer to the regularity criteria under one component only in the case $p=2$ as mentioned in \cite{Isab}.}
\begin{align*}
I(u)&\lesssim \norm {\nabla_h u^3}_{\dot{H}^{\frac{2}{p}-\frac{1}{2}}} \norm {\partial_j u^k \partial_{\ell} u^m}_{\dot{H}^{\frac{1}{2}-\frac{2}{p}}}\\
&\lesssim \norm {u^3}_{\dot{H}^{\frac{2}{p}+\frac{1}{2}}} \norm {\nabla_hu}_{\dot{H}^{1-\frac{1}{p}}}^2\\
&\lesssim \norm {u^3}_{\dot{H}^{\frac{2}{p}+\frac{1}{2}}} \norm {\nabla_hu}_{L^2}^{\frac{2}{p}}\norm {\nabla_hu}_{\dot{H}^1}^{2-\frac{2}{p}}.
\end{align*}
In particular for $p=\infty$ we have:
\begin{equation}\label{I(u)estimate}
I(u)\lesssim \norm {u^3}_{H^{\frac{1}{2}}} \norm {\nabla_hu}_{H^1}^{2}.
\end{equation}
The term $\mathcal{E}_3(u)$, can be estimated by using lemma \ref{hardtermlemma}, to obtain
\begin{equation*}
 \mathcal{E}_3 (u)\leq \bigg(\frac{1}{10}+ C \norm {u^3}_{H^{\frac{1}{2}}_{log_v,E}} \bigg) \norm {\nabla_h u}_{\dot{H}^1}^2 	+C \norm {u^3}_{\dot{H}^{\frac{1}{2}}} ^2\frac{\norm{\partial_3 u^h}_{L^2}^2}{E^2}.
\end{equation*}
We define then
\begin{equation*}
T_*\overset{\text{def}}{=}\sup \bigg\{ T\in [0,T^*[ / \sup_{t\in[0,T]}\norm {u^3(t)}_{H^{\frac{1}{2}}_{log_v,E}} \leq \frac{1}{4C}  \bigg\}.
\end{equation*}
Therefore, for all $t\leq T_*$, relation \eqref{innerproduct} together with estimate \eqref{I(u)estimate}, lemma \ref{hardtermlemma} and the classical $L^2-$energy estimate lead to
\begin{equation}\label{uniformbound}
 \norm {\nabla_h u(t) }_{L^2}^2 +\int_0^t \norm {\nabla_h u(s)}_{\dot{H}^1}^2ds \leq \norm {\nabla_h u_0}_{L^2}^2 + \frac{\norm {u_0}_{L^2}^2}{E^2}.
\end{equation}
On the other hand, as explained in \cite{Isab}, multiplying $(NS)$ by $-\partial_3^2 u$, integrating over $\mathbb{R}^3$, integration by parts together with the divergence free condition lead to
\begin{align*}
\frac{1}{2}\frac{d}{dt} \norm {\partial_3 u }_{L^2}^2 + \norm {\partial_3 u}_{\dot{H}^1}^2 &\lesssim \norm {\partial_3 u}_{L^6}\norm {\nabla_h u}_{L^3}\norm {\partial_3 u}_{L^2} \\
&\lesssim \frac{1}{2} \norm {\partial_3 u}_{\dot{H}^1}^2 + C \norm {\nabla_h u}_{L^2}\norm {\nabla_h u}_{\dot{H}^1} \norm {\partial_3 u }_{L^2}^2.
\end{align*}
\eqref{uniformbound} above leads then to a bound for $u$ in $L^{\infty}_{T_*}(\dot{H}^1)$.\\
Thus, by contraposition, if the quantity $\norm {u(t)}_{\dot{H}^1}$ blows up at a finite time $T^*>0$, then
\begin{equation*}
\forall t\in  [0,T^*[: \sup_{s\in[0,t]} \norm {u^3(s)}_{\dot{H}^{\frac{1}{2}}} >c_0 \overset{\text{def}}{=} \frac{1}{4C},
\end{equation*} 
which gives the desired result by passing to the limit $t\rightarrow T^*$.\\
Theorem \ref{log_vcase} is proved.
\end{proof}
\subsection{\textit{Proof of Theorem \ref{anisotropicbesovthm}}} 
\begin{proof}[]
Following for example an idea from \cite{Zujin1}, we multiply $(NS)$ by $-\Delta u$ and we integrate in space to obtain
\begin{align*}
\frac{1}{2}\frac{d}{dt} \norm {\nabla u}_{L^2}^2 + \norm {\Delta u }_{L^2}^2 &= \int_{\mathbb{R}^3} \big(u \cdot \nabla u \big) \cdot \Delta u  \\
& = - \int_{\mathbb{R}^3} \sum_{i,j,k=1}^3 \partial_k u^j\partial_j u^i \partial_k u^i.
\end{align*}  
For the time being, we don't know how to deal with the tri-linear term on the right hand side above in order to obtain a global-estimate of $u$ in $ L^{\infty}_T \dot{H}^1_x \cap L^2_T \dot{H}^2_x$. So to close the estimates the idea is similar to the one in Theorems \ref{log_hcase} and \ref{log_vcase}, and it consists in looking at this term as a bi-linear operator acting on $ \big(L^{\infty}_T \dot{H}^1_x \cap L^2_T \dot{H}^2_x\big)^2$ after assuming a condition which allows to control some components of the matrix $\partial_i u^j$.\\
Let us recall the Biot-Savart law identity which allows to write the so-called div-curl decomposition of $u^h$ as
\begin{equation}\label{u^h}
u^h=\nabla^{\perp}_h\Delta_h^{-1}(\omega^3)-\nabla_h \Delta^{-1}_h(\partial_3 u^3).
\end{equation}
Identity \eqref{u^h} insures that, for $(i,j)\in\{1,2\}^2$, $\partial_i u^j$ can be writing in terms of $\omega^3$ and $\partial_3 u^3$, modulo some anisotropic Fourier-multipliers of order zero, more precisely we have, for $(i,j)\in \{ 1,2 \}^2$
\begin{equation*}
\partial_i u^j = \mathcal{R}_{i,j} \omega_3 + \tilde{\mathcal{R}}_{i,j} \partial_3 u^3,
\end{equation*}
where $\tilde{\mathcal{R}}_{i,j}$ and $ \mathcal{R}_{i,j}$ are zero-order Fourier multipliers bounded from $L^q$ into $L^q$ for all $q$ in $]1,\infty[$.
On the other hand, the quantity $\partial_k u^j\partial_j u^i \partial_k u^i$ contains always, at least, one term of the form $\partial_i u^j$ with $(i,j)\in \{1,2\}^2$ or $i=j=3$, we infer that
\begin{align}\label{estimates-to-use-in-both-theorems}
\frac{1}{2}\frac{d}{dt} \norm {\nabla u}_{L^2}^2 + \norm {\Delta u }_{L^2}^2 \lesssim \sum_{ \underset{(i,j)\in \{1,2\}}{(l,k,m,n)\in \{1,2,3\}}}\bigg|\int_{\mathbb{R}^3}\big(\mathcal{R}_{i,j} \omega_3 + \tilde{\mathcal{R}}_{i,j} \partial_3 u^3 \big). \partial_k u^l \partial_m u^n\bigg|
\end{align}  
Lemma \ref{lemmafortheorems2} gives then 
\begin{equation*}
\frac{1}{2}\frac{d}{dt} \norm {\nabla u}_{L^2}^2 + \norm {\nabla u }_{\dot{H}^1}^2 \leq  \frac{1}{10} \norm {\nabla u}_{\dot{H}^1}^2 + C \bigg( \norm {\partial_3 u_3}_{\dot{B}^{\alpha, s_p-\alpha}_{2,\infty}}^p + \norm {\omega^3}_{\dot{B}^{\beta, s_m-\beta}_{2,\infty}}^m\bigg) \norm {\nabla u}_{L^2}^2.
\end{equation*}
Gronwall lemma leads then to
\begin{equation}\label{global estimate anisotropic}
 \norm {\nabla u(t)}_{L^2}^2 + \int_0^t\norm {\nabla u(t') }_{\dot{H}^1}^2 dt' \lesssim \norm {\nabla u_0}_{L^2}exp \bigg[C \int_0^t\bigg( \norm {\partial_3 u_3 (t')}_{\dot{B}^{\alpha, s_p-\alpha}_{2,\infty}}^p + \norm {\omega^3(t')}_{\dot{B}^{\beta, s_m-\beta}_{2,\infty}}^m\bigg) dt'\bigg].
\end{equation}
That is if, for some $\alpha,\beta,p,m$ satisfying the hypothesis of Theorem \ref{anisotropicbesovthm}, the quantity in the right hand side of \eqref{global estimate anisotropic} is finite, then $u$ is bounded in $L^{\infty}_T(\dot{H}^1)$. By contraposition, if there is a blow-up of the $\dot{H}^1$ norm at some finite $T^*$ then, for all $\alpha,\beta,p,m$
\begin{align*}
\int_0^{T^*}\bigg( \norm {\partial_3 u_3 (t')}_{\dot{B}^{\alpha, s_p-\alpha}_{2,\infty}}^p + \norm {\omega^3(t')}_{\dot{B}^{\beta, s_m-\beta}_{2,\infty}}^m\bigg) dt' = \infty.
\end{align*}
Theorem \ref{anisotropicbesovthm} is proved.
\end{proof}
\subsection{\textit{Proof of Theorem \ref{negativebesovthm}}}
\begin{proof}[]
The proof of Theorem \ref{negativebesovthm} doesn't differ a lot from the previous one. We restart from \eqref{estimates-to-use-in-both-theorems}, applying lemma \ref{lemmafortheorems3} gives\footnote{Note that the case $q_i=\infty$ is included in the estimates proved in Lemma \ref{lemmafortheorems2}, however we did not say anything about this case in Theorem \ref{negativebesovthm} due to the lack of continuity of Riesz operators $\mathcal{R}_{i,j}$ and $\tilde{\mathcal{R}}_{i,j}$ from $L^{\infty}$ into $L^{\infty}$. }
\begin{equation*}
\frac{1}{2}\frac{d}{dt} \norm {\nabla u}_{L^2}^2 + \norm {\nabla u }_{\dot{H}^1}^2 \leq  \frac{1}{10} \norm {\nabla u}_{\dot{H}^1}^2 + C \bigg( \norm {\partial_3 u_3}_{\mathcal{B}_{q_1,p_1} }^{p_1} + \norm {\omega^3}_{\mathcal{B}_{q_2,p_2} }^{p_2}\bigg) \norm {\nabla u}_{L^2}^2.
\end{equation*}
Next, integrating in time interval $[0,t]$, and applying Gronwall lemma gives
\begin{equation*} 
 \norm {\nabla u(t)}_{L^2}^2 + \int_0^t\norm {\nabla u(t') }_{\dot{H}^1}^2 dt' \lesssim \norm {\nabla u_0}_{L^2}exp \bigg[C \int_0^t\bigg(  \norm {\partial_3 u_3}_{\mathcal{B}_{q_1,p_1} }^{p_1} + \norm {\omega^3}_{\mathcal{B}_{q_2,p_2} }^{p_2}\bigg) dt'\bigg].
\end{equation*}
Same arguments as in the conclusion of the previous theorem lead to the desired result.\\
Theorem \ref{negativebesovthm} is proved.
\end{proof}
\section{Proof of the three lemmas}
\subsection{\textit{Proof of lemma \ref{hardtermlemma}}}
\begin{proof} [] 
Let us recall a definition from \cite{Isab}. For $E\in \mathbb{R}^+$ and $a\in \mathcal{S}'(\mathbb{R}^3)$:
\begin{align*}
a_{\flat, E^{-1}}\overset{def}{=} \mathcal{F}^{-1}(\mathbf{1}_{B_h(0,E^{-1})} \hat{a}), \; \; a_{\sharp, E^{-1}}\overset{def}{=}\mathcal{F}^{-1}(\mathbf{1}_{B_h^c(0,E^{-1})} \hat{a}).
\end{align*}
Based on this decomposition, we write
\begin{align*}
J_{i\ell}(u,u^3)=J_{ E}^{\flat} + J_{ E}^{\sharp},
\end{align*}
where
\begin{align*}
J_{ E}^{\flat}\overset{\text{def}}{=}  \int_{\mathbb{R}^3} \partial_i u^3 \partial_3 u^{\ell}_{\flat,E^{-1}} \partial_i u^{\ell}, \text{and } \;  
J_{ E}^{\sharp}\overset{\text{def}}{=}  \int_{\mathbb{R}^3} \partial_i u^3 \partial_3 u^{\ell}_{\sharp,E^{-1}} \partial_i u^{\ell} .
\end{align*}
The main point consists in estimating $J_{ E}^{\sharp}$. Using Bony's decomposition with respect to the horizontal variables, to write
\begin{align*}
&J_{ E}^{\sharp}= J_{ E}^{\sharp,1} + J_{ E}^{\sharp,2} \;\; \; \text{with}\\
&J_{ E}^{\sharp,1}\overset{\text{def}}{=}  \displaystyle\int_{\mathbb{R}_v} \bigg( \int_{\mathbb{R}_h^2} \partial_i u^{\ell}(x_h,x_3) \tilde{T}^h_{\partial_i u^3(x_h,x_3)} \partial_3 u^{\ell}_{\sharp,E^{-1}}(x_h,x_3)  dx_h\bigg)dx_3\\
&J_{ E}^{\sharp,2}\overset{\text{def}}{=}  \displaystyle\int_{\mathbb{R}_v} \bigg( \sum_{k\in \mathbb{Z}}\int_{\mathbb{R}_h^2} \Delta_k^h\partial_i u^{\ell}(x_h,x_3) \tilde{\Delta}_k^h T^h_{\partial_3 u^{\ell}_{\sharp,E^{-1}}(x_h,x_3)}   \partial_i u^3(x_h,x_3)  dx_h\bigg)dx_3.
\end{align*}
$J_{ E}^{\sharp,1}$ can be estimated by duality then by using some product laws (lemma \ref{produclaw}), we obtain
\begin{align*}
J_{ E}^{\sharp,1} & \lesssim \norm {\nabla_h u}_{L^{\infty}_v(\dot{H}^{\frac{1}{2}}_h )} \norm { \tilde{T}^h_{\partial_i u^3} \partial_3 u^{\ell}_{\sharp,E^{-1}}}_{L^{1}_v(\dot{H}^{-\frac{1}{2}}_h )} \\
& \lesssim \norm {\nabla_h u}_{L^{\infty}_v(\dot{H}^{\frac{1}{2}}_h )}  \norm {\nabla_h u^3}_{L^2_v(\dot{H}^{-\frac{1}{2}}_h)} \norm {\partial_3 u}_{L^2_v(\dot{H}^{1}_h)}.
\end{align*}
Using then the inequality: $\norm {\nabla_h u}_{L^{\infty}_v(\dot{H}^{\frac{1}{2}}_h)} \lesssim \norm {\nabla_h u}_{\dot{H}^1(\mathbb{R}^3)}$ (see lemma \ref{lemma-from-isab}), we infer that
\begin{equation}\label{j1}
J_{ E}^{\sharp,1} \lesssim \norm {u^3}_{\dot{H}^{\frac{1}{2}}(\mathbb{R}^3)}\norm {\nabla_h u}_{\dot{H}^1(\mathbb{R}^3)}^2.
\end{equation}
In order to estimate $J_{ E}^{\sharp,2}$ we split it into a sum of a good term $J_{ E}^{\sharp,2,\textbf{G}}$ and a bad one $J_{ E}^{\sharp,2,\textbf{B}}$ based on the dominated frequencies of $\partial_3 u^{\ell}$
\begin{align*}
&\partial_3 u^{\ell}_{\sharp,E^{-1}} = \partial_3 u^{\ell,\textbf{G}}_{\sharp,E^{-1}} + \partial_3 u^{\ell,\textbf{B}}_{\sharp,E^{-1}} \;\;\; \text{with}\\
& \partial_3 u^{\ell,\textbf{G}}_{\sharp,E^{-1}} \overset{\text{def}}{=} \sum_{q \leq k} \Delta_k^h \Delta_q^v \partial_3 u^{\ell}_{\sharp,E^{-1}} \;\;\; \text{ and } \;\;\; \partial_3 u^{\ell,\textbf{B}}_{\sharp,E^{-1}} \overset{\text{def}}{=} \sum_{k < q} \Delta_k^h \Delta_q^v \partial_3 u^{\ell}_{\sharp,E^{-1}}.
\end{align*}
The good term can be easily estimated without using the fact that $u^{\ell}_{\sharp,E^{-1}}$ contains only the high horizontal frequencies, but only providing that the horizontal frequencies control the vertical ones. We proceed as follows, by using the product lemma \ref{produclaw} we find:
\begin{align*}
J_{ E}^{\sharp,2,\textbf{G}} \lesssim \norm {\nabla_h u^3}_{\dot{H}^{-\frac{1}{2}}_h(L^2_v)} \norm {\partial_3 u^{\ell,\textbf{G}}_{\sharp,E^{-1}}}_{\dot{H}^{\frac{3}{4}}_h(\dot{H}^{\frac{1}{4}}_v)}\norm {\nabla_h u}_{\dot{H}^{\frac{3}{4}}_h(\dot{H}^{\frac{1}{4}}_v)}.
\end{align*}
Lemma \ref{goodfrequencieslemma} in Appendix gives then
$$J_{ E}^{\sharp,2,\textbf{G}} \lesssim \norm {\nabla_h u^3}_{\dot{H}^{-\frac{1}{2}}_h(L^2_v)} \norm {\nabla_h u}_{\dot{H}^{\frac{3}{4}}_h(\dot{H}^{\frac{1}{4}}_v)}^2$$
which yields finally, by using lemma \ref{embedding-anisotropic} 
\begin{equation}\label{j2g}
J_{ E}^{\sharp,2,\textbf{G}} \lesssim \norm {u^3}_{\dot{H}^{\frac{1}{2}}} \norm {\nabla_h u}_{\dot{H}^1}^2
\end{equation}
In order to estimate the bad term $J_{ E}^{\sharp,2,\textbf{B}}$, we use the Bony's decomposition with respect to vertical variables to infer that
\begin{align}\label{J2B}
J_{ E}^{\sharp,2,\textbf{B}} \lesssim \sum_{q,k\in \mathbb{Z}}\norm {\Delta_k^h \Delta_q^v \nabla_h u}_{L^2(\mathbb{R}^3)} \big( \mathcal{I}_{k,q}^{(1)}+\mathcal{I}_{k,q}^{(2)}+\mathcal{I}_{k,q}^{(3)} \big),
\end{align}
where
\begin{align*}
\mathcal{I}_{k,q}^{(1)}&\overset{\text{def}}{=}\norm { S_{k-1}^hS_{q-1}^v (\partial_3 u^{\ell,\textbf{B}}_{\sharp,E^{-1}} )}_{L^{\infty}_v(L^{\infty}_h)}\norm {\Delta_k^h \Delta_q^v \nabla_h u^3}_{L^2(\mathbb{R}^3)}  \\
\mathcal{I}_{k,q}^{(2)}&\overset{\text{def}}{=} \norm { S_{k-1}^h\Delta_q^v (\partial_3 u^{\ell,\textbf{B}}_{\sharp,E^{-1}} )}_{L^2_v(L^{\infty}_h)}  \norm {\Delta_k^h S_{q-1}^v \nabla_h u^3}_{L^{\infty}_v(L^2_h)} \\
\mathcal{I}_{k,q}^{(3)}&\overset{\text{def}}{=} 2^{\frac{q}{2}}\sum_{j\geq q - N_0} \norm { S_{k-1}^h\Delta_j^v (\partial_3 u^{\ell,\textbf{B}}_{\sharp,E^{-1}}) }_{L^2_v(L^{\infty}_h)}  \norm {\Delta_k^h \tilde{\Delta}_j^v \nabla_h u^3}_{L^2(\mathbb{R}^3)} .
\end{align*}
The estimates of these terms are based on lemma \ref{badfrequencieslemma} proved in Appendix, by taking $ f^{\textbf{B}}_{\sharp,E^{-1}}= \partial_3 u^{\ell,\textbf{B}}_{\sharp,E^{-1}} $.\\
We use inequality \eqref{bad2} from lemma \ref{badfrequencieslemma} to estimate $\mathcal{I}_{k,q}^{(1)}$, which gives
\begin{align*}
\mathcal{I}_{k,q}^{(1)}& \lesssim \big( log(E2^q+e)\big)^{\frac{1}{2}} c_q 2^{\frac{q}{2}} \norm {\nabla_h \partial_3 u}_{L^2(\mathbb{R}^3)}2^{k}\norm {\Delta_k^h \Delta_q^v  u^3}_{L^2(\mathbb{R}^3)}  \\
&\lesssim  c_q^2 c_k 2^{\frac{q}{2}} 2^{\frac{k}{2}} \norm {\nabla_h \partial_3 u}_{L^2(\mathbb{R}^3)} \norm {u^3}_{\dot{H}^{\frac{1}{2}}_{log_{v},E}}.
\end{align*}
Finally we obtain
\begin{equation}\label{I_1}
\mathcal{I}_{k,q}^{(1)} \lesssim c_k 2^{\frac{q}{2}} 2^{\frac{k}{2}} \norm {\nabla_h \partial_3 u}_{L^2(\mathbb{R}^3)} \norm {u^3}_{\dot{H}^{\frac{1}{2}}_{log_{v},E}}.
\end{equation}
In order to estimate $\mathcal{I}_{k,q}^{(2)}$ we use inequality \eqref{bad1}, we infer that
\begin{align*}
\mathcal{I}_{k,q}^{(2)}& \lesssim \big( log(E2^q+e)\big)^{\frac{1}{2}} c_q  \norm {\nabla_h \partial_3 u}_{L^2(\mathbb{R}^3)}2^{k}\norm {\Delta_k^h S_{q-1}^v  u^3}_{L^{\infty}_v(L^2_h)} \\
&\lesssim  c_q 2^{\frac{k}{2}} \norm {\nabla_h \partial_3 u}_{L^2(\mathbb{R}^3)} \big( log(E2^q+e)\big)^{\frac{1}{2}} 2^{\frac{k}{2}}\norm {\Delta_k^h S_{q-1}^v u^3}_{L^{\infty}_v(L^2_h)}.
\end{align*}
Next, we use the following estimate
\begin{align*}
2^{-\frac{q}{2}} \big( log(E2^q+e)\big)^{\frac{1}{2}} 2^{\frac{k}{2}}\norm {\Delta_k^h S_{q-1}^v u^3}_{L^{\infty}_v(L^2_h)} &\lesssim \sum_{m\leq q} \big(  2^{(m-q)}  log(E2^q+e) \big)^{\frac{1}{2}} 2^{\frac{k}{2}}\norm {\Delta_k^h \Delta_m^v u^3}_{L^2(\mathbb{R}^3)},
\end{align*}
together with the fact that
\begin{align*}
log(E2^q+e) &\leq log(2^{q-m}(E2^m+e)), \;\;\; \forall m\leq q\\
& \leq log(E2^m+e) + (q-m)\\
&\lesssim log(E2^m+e)(1+ (q-m)).
\end{align*}
This leads to
\begin{align*}
2^{-\frac{q}{2}} \big( log(E2^q+e)\big)^{\frac{1}{2}} 2^{\frac{k}{2}}\norm {\Delta_k^h S_{q-1}^v u^3}_{L^{\infty}_v(L^2_h)} &\lesssim \sum_{m\leq q} (\sigma_{q-m} c_m) c_k \norm {u^3}_{\dot{H}^{\frac{1}{2}}_{log_{v},E}},
\end{align*}
where $\sigma_j \bydef \displaystyle\frac{\sqrt{1+j}}{2^{\frac{j}{2}}} \in \ell_j^1(\mathbb{N}) $. By using convolution inequality, we deduce that
\begin{equation}\label{I_2}
\mathcal{I}_{k,q}^{(2)} \lesssim   c_k 2^{\frac{k}{2}} 2^{\frac{q}{2}} \norm {\nabla_h \partial_3 u}_{L^2(\mathbb{R}^3)} \norm {u^3}_{\dot{H}^{\frac{1}{2}}_{log_{v},E}} .
\end{equation}
Finally, in order to estimate $\mathcal{I}_{k,q}^{(3)}$, we use again inequality \eqref{bad2} from lemma \ref{badfrequencieslemma} bellow, we obtain
\begin{align*}
\mathcal{I}_{k,q}^{(3)} &\lesssim 2^{\frac{q}{2}}2^{\frac{k}{2}}\bigg(\sum_{j\geq q-N_0}  \big( log(E2^j+e)\big)^{\frac{1}{2}} c_j 2^{\frac{k}{2}}\norm {\Delta_k^h \Delta_{j}^v  u^3}_{L^2(\mathbb{R}^3)} \bigg) \norm {\nabla_h \partial_3 u}_{L^2(\mathbb{R}^3)} \\
&\lesssim c_k 2^{\frac{q}{2}}2^{\frac{k}{2}} \norm {\nabla_h \partial_3 u}_{L^2(\mathbb{R}^3)} \norm {u^3}_{\dot{H}^{\frac{1}{2}}_{log_{v},E}} .
\end{align*} 
Together with \eqref{I_1} and \eqref{I_2} yield
\begin{equation*}
\sum_{i\in \{1,2,3 \}} \mathcal{I}_{k,q}^{(i)}  \lesssim c_k 2^{\frac{q}{2}}2^{\frac{k}{2}} \norm {\nabla_h u}_{H^1(\mathbb{R}^3)} \norm {u^3}_{\dot{H}^{\frac{1}{2}}_{log_{v},E}}
\end{equation*}
Plugging this last one into \eqref{J2B} gives
\begin{align*}
J_{ E}^{\sharp,2,\textbf{B}} &\lesssim  \bigg(\sum_{k,q\in \mathbb{Z}} c_k 2^{\frac{q}{2}}2^{\frac{k}{2}} \norm {\Delta_k^h \Delta_q^v \nabla_h u}_{L^2(\mathbb{R}^3)} \bigg)  \norm {\nabla_h u}_{\dot{H}^1(\mathbb{R}^3)} \norm {u^3}_{\dot{H}^{\frac{1}{2}}_{log_{v},E}}\\
& \lesssim \norm {\nabla_h u}_{\dot{H}^{\frac{1}{2}}_h(\dot{B}^{\frac{1}{2}}_{2,1})_v} \norm {\nabla_h u}_{\dot{H}^1(\mathbb{R}^3)} \norm {u^3}_{\dot{H}^{\frac{1}{2}}_{log_{v},E}}.
\end{align*}
Lemma \ref{embedding-anisotropic} then gives
\begin{equation}\label{j2b}
J_{ E}^{\sharp,2,\textbf{B}} \lesssim  \norm {\nabla_h u}_{\dot{H}^1(\mathbb{R}^3)}^2 \norm {u^3}_{\dot{H}^{\frac{1}{2}}_{log_{v},E}}.
\end{equation}
From \eqref{j1}, \eqref{j2g} and \eqref{j2b} we deduce
\begin{equation*}
J_{ E}^{\sharp} \lesssim  \norm {\nabla_h u}_{\dot{H}^1(\mathbb{R}^3)}^2 \norm {u^3}_{\dot{H}^{\frac{1}{2}}_{log_{v},E}}.
\end{equation*}
$J_{ E}^{\flat}$ can be estimated along the same lines as in \cite{Isab}, by using the product law $(\dot{B}^{1}_{2,1})_h \times \dot{H}^{\frac{1}{2}}_h \subset \dot{H}^{\frac{1}{2}}_h$, together with the embedding $\dot{H}^1(\mathbb{R}^3) \hookrightarrow L^{\infty}_v(\dot{H}^{\frac{1}{2}}_h)$ (see lemma \ref{lemma-from-isab} in Appendix), we infer that
\begin{align*}
J_{ E}^{\flat} &\lesssim \norm {\nabla_h u^3}_{L^2_v(\dot{H}^{-\frac{1}{2}})_h}\norm {\partial_3 u^{\ell}_{\flat, E^{-1}} \partial_i u^{\ell}}_{L^2_v(\dot{H}^{\frac{1}{2}}_h)}\\
&\lesssim \norm {u^3}_{\dot{H}^{\frac{1}{2}}}\norm {\partial_3 u^{\ell}_{\flat, E^{-1}}}_{L^2_v(\dot{B}^{1}_{2,1})_h} \norm {\nabla_h u}_{L^{\infty}_v(\dot{H}^{\frac{1}{2}}_h )}\\
&\lesssim \norm {u^3}_{\dot{H}^{\frac{1}{2}}}\norm {\nabla_h u}_{\dot{H}^1}\frac{\norm {{\partial_3 u}}_{L^2}}{E} \\
&\lesssim \frac{1}{100} \norm {\nabla_h u}_{\dot{H}^1} + C \norm {u^3}_{\dot{H}^{\frac{1}{2}}}^2 \frac{\norm {{\partial_3 u}}_{L^2}^2}{E^2}.
\end{align*}
Lemma \ref{hardtermlemma} is then proved.
\end{proof}
\subsection{\textit{Proof of lemma \ref{lemmafortheorems3}}}
\begin{proof}[]
Let $p\in [2,4]$ and $\alpha\in \big[0,\frac{2}{p}-\frac{1}{2}\big] $. We define $q$ and $\theta$ such that
\begin{equation}\label{numerology2}
\frac{2}{q}\bydef 1-\frac{1}{p},
\end{equation}
\begin{equation}\label{numerology3}
\theta \bydef \frac{1}{2} + \frac{\alpha}{2} - \frac{1}{p}.
\end{equation}
One may check that 
\begin{align*}
q \in ]2,\infty[ \; \; \; \text{and} \; \; \; \theta \in \bigg[0, \frac{2}{q}\bigg] \cap \bigg[0,\frac{1}{2}\bigg],
\end{align*}
which allow us to use the following embedding, due to lemmas \ref{embedding-anisotropic} and \ref{interpolation-lemma}
\begin{equation}\label{embedding5}
L^{\infty}_T (L^2(\mathbb{R}^3) )\cap  L^{2}_T (\dot{H}^1(\mathbb{R}^3) ) \hookrightarrow L^q_T(\dot{B}^{\frac{2}{q}}_{2,1}(\mathbb{R}^3)) \hookrightarrow L^q_T\big( (\dot{B}^{\frac{2}{q}- \theta}_{2,1})_h(\dot{B}^{ \theta}_{2,1})_v \big) .
\end{equation}
Thus, by using lemma \ref{produclaw}, if $g\in (\dot{B}^{\frac{2}{q}- \theta}_{2,1})_h(\dot{B}^{ \theta}_{2,1})_v$ then $g\cdot g \in (\dot{B}^{\frac{4}{q}- 2\theta - 1}_{2,1})_h(\dot{B}^{ 2\theta-\frac{1}{2}}_{2,1})_v$.\\
By virtue of \eqref{numerology2}, \eqref{numerology3} and embedding \eqref{embedding5}, we infer that
\begin{equation*}
\norm {g\cdot g}_{(\dot{B}^{-\alpha}_{2,1})_h(\dot{B}^{ -\frac{2}{p}+\frac{1}{2}+\alpha }_{2,1})_v} \lesssim \norm g_{(\dot{B}^{\frac{2}{q}- \theta}_{2,1})_h(\dot{B}^{ \theta}_{2,1})_v}^2,
\end{equation*}
which gives by duality, embedding \eqref{embedding5} and lemma \ref{interpolation-lemma} 
\begin{align*}
\big| \big( fg|g  \big)_{L^2} \big|& \lesssim \norm f_{(\dot{B}^{ \alpha}_{2,\infty})_h(\dot{B}^{  \frac{2}{p}-\frac{1}{2}-\alpha }_{2,\infty})_v} \norm {g\cdot g}_{(\dot{B}^{-\alpha}_{2,1})_h(\dot{B}^{ -\frac{2}{p}+\frac{1}{2}+\alpha }_{2,1})_v}\\
&\lesssim \norm f_{(\dot{B}^{ \alpha}_{2,\infty})_h(\dot{B}^{ s_p-\alpha }_{2,\infty})_v} \norm g_{\dot{B}^{\frac{2}{q}}_{2,1}}^2\\
&\lesssim \norm f_{(\dot{B}^{ \alpha}_{2,\infty})_h(\dot{B}^{ s_p-\alpha }_{2,\infty})_v} \norm g_{L^2}^{\frac{2}{p}} \norm g_{\dot{H}^1}^{2(1-\frac{2}{p})}.
\end{align*}
Finally we obtain
\begin{align*}
\big| \big( fg|g  \big)_{L^2} \big| \leq \frac{1}{10}  \norm g_{\dot{H}^1}^2 + C\norm f_{(\dot{B}^{ \alpha}_{2,\infty})_h(\dot{B}^{ s_p-\alpha }_{2,\infty})_v} ^p\norm g_{L^2}^2.
\end{align*}
Lemma \ref{lemmafortheorems3} is proved.
\end{proof}
\subsection{\textit{Proof of lemma \ref{lemmafortheorems2}}}
\begin{proof}[]
According to lemma \ref{interpolation-lemma} in Appendix, in particular inequality \eqref{ineterpolation-inequa} gives
\begin{equation}\label{ineterpolation-inequa2}
\norm {g(t,.)}_{\dot{B}^{\frac{2}{m}}_{2,1}(\mathbb{R}^3))} \lesssim  \norm {g(t,.)}_{\dot{H}^1(\mathbb{R}^3)}^{\frac{2}{m}} \norm {g(t,.)}_{L^2(\mathbb{R}^3)}^{1-\frac{2}{m}},\;\; \forall m\in ]2,\infty[.
\end{equation} 
We use then the Bony's decomposition to study the product $g\cdot g$. \\
Let $(q,p)\in [1,\infty]^2$ satisfying
\begin{equation*}
q\in [3,\infty] \;\;\; \text{ and }\; \; \; \frac{3}{q}+\frac{2}{p} \in ]1,2[.
\end{equation*}
Let $(m_1, m_2)$ be in $[2,\infty] \times ]2,\infty[, $ given by
\begin{equation}\label{numerologiy}
\frac{2}{3m_1} \bydef \frac{1}{q} \;\;\; \text{ and }\; \; \;  2\big( 1-\frac{1}{m_2} \big)\bydef \frac{3}{q}+ \frac{2}{p}\in ]1,2[\iff m_2 \in ]2,\infty[.
\end{equation}
Let us define the real number $N_{m_1}$ associated to the embedding $\dot{H}^{\frac{2}{{m_1}}}(\mathbb{R}^3)$ in $L^{N_{m_1}}(\mathbb{R}^3)$
\begin{equation*}
\frac{1}{N_{m_1}}\bydef \frac{1}{2} -\frac{2}{3m_1} \in \bigg [ \frac{1}{6}, \frac{1}{2} \bigg ] .
\end{equation*}
Let us also define $r$ to be the conjugate of $q$, that is
\begin{equation*}
\frac{1}{r} \bydef 1-\frac{1}{q} \in \bigg[ \frac{2}{3}, 1  \bigg ].
\end{equation*}
We write
$$\Delta_j (g\cdot g)= 2\Delta_j T_g(g) +\Delta_j R(g,g)$$
where $T$ and $R$ are the operators associated to the Bony's decomposition, defined in the Appendix.\\
We turn now to estimate the two parts of $\Delta_j(g\cdot g)$. We have
\begin{align*}
\norm {\Delta_j T_g(g)}_{L^{r }} &\lesssim \norm {S_{j-1} g}_{L^{N_{m_1}}} \norm {\Delta_j g}_{L^2}\\
&\lesssim \norm {g}_{L^{N_{m_1}}} d_j 2^{-j\frac{2}{m_2}} \norm {g}_{\dot{B}^{\frac{2}{m_2}}_{2,1}} ,
\end{align*}
using then the embedding
\begin{equation}\label{embedding1}
\norm {g}_{L^{N_{m_1}}} \lesssim \norm g_{\dot{H}^{\frac{2}{m_1}}},
\end{equation}
together with the interpolation inequality \eqref{ineterpolation-inequa2} gives
\begin{align}\label{T_gg}
\norm {\Delta_jT_g(g)}_{\dot{B}^{\frac{2}{m_2}}_{r ,1}} &\lesssim  2^{-j\frac{2}{m_2}} d_j  \norm g_{\dot{H}^1}^{\frac{2}{m_1}+\frac{2}{m_2}} \norm {g}_{L^2}^{2-\frac{2}{m_1}+\frac{2}{m_2} }.
\end{align}
For the remainder term, we proceed almost similarly
\begin{align*}
\norm {\Delta_j R(g,g)}_{L^{r }} &\lesssim \sum_{j'\geq j-5} \norm {\tilde{\Delta}_{j'} g}_{L^{N_{m_1}}} \norm {\Delta_j' g}_{L^2}\\
&\lesssim 2^{-j\frac{2}{m_2}}\sum_{j'\geq j-5}\big( d_{j'} 2^{-(j'-j)\frac{2}{m_2}}\big)  \norm {g}_{L^{N_{m_1}}}\norm {g}_{\dot{B}^{\frac{2}{m_2}}_{2,1}}.
\end{align*}
Where $\tilde{\Delta_j'}\bydef \displaystyle\sum_{i\in \{-1,0,1 \}} \Delta_{j'+i}$.
By convolution inequality, interpolation inequality \eqref{ineterpolation-inequa2} and the embedding one \eqref{embedding1}, we get
\begin{equation*}
\norm {\Delta_j R(g,g)}_{L^{r }} \lesssim  2^{-j\frac{2}{m_2}} d_j  \norm g_{\dot{H}^1}^{\frac{2}{m_1}+\frac{2}{m_2}} \norm {g}_{L^2}^{2-\frac{2}{m_1}+\frac{2}{m_2} },
\end{equation*}
which gives, together with \eqref{T_gg}
\begin{equation*}
\norm {\Delta_j (g \cdot g)}_{L^{r }}\lesssim  2^{-j\frac{2}{m_2}} d_j  \norm g_{\dot{H}^1}^{\frac{2}{m_1}+\frac{2}{m_2}} \norm {g}_{L^2}^{2-\frac{2}{m_1}+\frac{2}{m_2} }.
\end{equation*}
On the other hand, by duality, we get
\begin{align*}
\big| \big( fg|g  \big)_{L^2} \big|& \lesssim \sum_{j\in \mathbb{Z}} \norm {\Delta_j f}_{L^{q}} \norm {\Delta_j (g \cdot g)}_{L^{r}}\\
& \lesssim \sum_{j\in \mathbb{Z}} \big(2^{-j\frac{2}{m_2}} \norm {\Delta_j f}_{L^{q}} d_j \big)  \norm g_{\dot{H}^1}^{\frac{2}{m_1}+\frac{2}{m_2}} \norm {g}_{L^2}^{2-\frac{2}{m_1}+\frac{2}{m_2} }  \\
&\lesssim \norm {f}_{\dot{B}_{q,\infty}^{-\frac{2}{m_2}}}\norm g_{\dot{H}^1}^{\frac{2}{m_1}+\frac{2}{m_2}} \norm {g}_{L^2}^{2-\frac{2}{m_1}+\frac{2}{m_2} }. 
\end{align*}
By virtue of \eqref{numerologiy} we have
\begin{equation*}
1- \bigg( \frac{1}{m_1}+\frac{1}{m_2} \bigg)= \frac{1}{p} \;\;\; \text{ and } \; \; \; -\frac{2}{m_2}= \frac{3}{q}+ \frac{2}{p}-2.
\end{equation*}
This gives
\begin{equation*}
\big| \big( fg|g  \big)_{L^2} \big| \leq \frac{1}{10} \norm g_{\dot{H}^1}^2 +C \norm {f}_{\mathcal{B}_{q,p}} ^{p} \norm g_{L^2}^2.
\end{equation*}
Lemma \ref{lemmafortheorems2} is proved.
\end{proof}
\appendix	
\section*{Appendix}
  \renewcommand{\thesection}{\Alph{section}}
  
\section{Functional framework} In this part we recall some notions and definitions used in the previous sections.\\
Let us first recall some notions of the Littlewood-Paley theory, the anisotropic Besov spaces used in this paper and some of their properties. The anisotropic version used here is crucial, for more details about that and for more applications one may see for instance \cite{Chemin1,Ifti1,Ifti2,Ifti3,Haroune,Paicu2}.\\

 Let $(\psi,\varphi)$ be a couple of smooth functions with value in $[0,1]$ satisfying: 
\begin{alignat*}{2}
&\text{Supp } \psi \subset \{ \xi \in \mathbb{R} : |\xi| \leq \frac{4}{3}\}, 
\quad &&\text{Supp } \varphi \subset \{\xi \in \mathbb{R} :\frac{3}{4} \leq |\xi| \leq \frac{8}{3} \} \\
&\psi(\xi) + \sum_{q\in \mathbb{N}} \varphi(2^{-q}\xi) = 1 \;\; \forall \xi \in \mathbb{R}, 
\quad 
&&\sum_{q\in \mathbb{Z}} \varphi(2^{-q}\xi) = 1 \;\; \forall \xi \in \mathbb{R}\backslash \{0\}.
\end{alignat*}
Let $a$ be a tempered distribution, $\hat{a}=\mathcal{F}(a)$ its Fourier transform and $\mathcal{F}^{-1}$ 
denotes the inverse of $\mathcal{F}$. We define the homogeneous dyadic blocks $\Delta_q$ by setting
\begin{alignat*}{2}
&\Delta^v_q a \bydef \mathcal{F}^{-1}\big(\varphi (2^{-q}|\xi_3| \hat{a} ) \big),\; \forall \; q\in \mathbb{Z}, 
\quad && \Delta^h_j a \bydef \mathcal{F}^{-1}\big(\varphi (2^{-j}|\xi_h| \hat{a} ) \big),\; \forall \; j\in \mathbb{Z}, \\
& S_q^v \bydef  \sum_{q'< q} \Delta _{q' }^v, \ \forall q \in \mathbb Z, \quad 
&& S_j^h \bydef  \sum_{j'< j} \Delta _{j' }^h, \ \forall j \in \mathbb Z.
\end{alignat*}
Moreover, in all the situations, i.e. for 
$\Delta,S$ with the same index of direction (horizontal or vertical) it holds:
\begin{align*}
&\Delta_m\Delta_{m'} a =0 \; \text{if} \; |m-m'| \geq 2 \\
&\Delta_m\big(S_{m'-1}a\Delta_{m'} a\big) =0 \; \text{if} \; |m-m'| \geq 5 \\
&\Delta_m\sum_{i\in \{0,1,-1 \}}\sum_{m'\in \mathbb{Z}}(\Delta_{m'+i}a\Delta_{m'} a\big)=\Delta_m\sum_{i\in \{0,1,-1 \}}\sum_{m'\geq m-5}(\Delta_{m'+i}a\Delta_{m'} a\big),
\end{align*} 
 We should recall the so-called Bony decomposition (see \cite{Chemin1})
\begin{alignat*}{2}
&ab= T_a(b) + T_b(a) + R(a,b), \quad &&\\
&T_a(b)\bydef \sum_{q\in Z} S_{q-1} a \Delta_q b, \quad 
&&R(a,b)\bydef \sum_{i\in \{0,1,-1 \}}\sum_{q\in Z} \Delta_{q+i} a \Delta_q b.
\end{alignat*}
It is also useful sometimes to use the following version
$$ab= \widetilde{T}_ab + T_ba,$$
where
$$\widetilde{T}_ab \bydef \sum_{q\in \mathbb{Z}} S_{q+2}a \Delta_q b.$$
Here again all the situations may be considered however particular cases must be precised by using the adequate 
notations. For instance if we consider the version for the vertical variable, we have to 
add the exponent $^{v}$ in all the operators $T_a,T_b,R,S_q$ and $\Delta_q$. \\ 

Next, we recall the definition of the anisotropic Besov spaces. See \cite{Chemin-Zhang} for more details.    
\begin{defin}
Let $s,t$ be two real numbers and let $p_1,p_2,q_1,q_2$ be in $[1,+\infty]$, we define the space $(\dot{B}^t_{p_1,q_1})_h(\dot{B}^s_{p_2,q_2})_v$ as the space of tempered distributions $u$ such that 
$$
\norm u _{(\dot{B}^t_{p_1,q_1})_h(\dot{B}^s_{p_2,q_2})_v}\bydef \norm { 2^{kt}2^{js} \norm {\Delta_k^h \Delta_j^v u}_{L^{p_1}_hL^{p_2}_v}}_{\ell_k^{q_1}(\mathbb{Z};\ell_j^{q_2}(\mathbb{Z})) } < \infty .
$$
In the situation where $q_1=q_2=q$ and $p_1=p_2=p$, we use the notation $\dot{B}_{p,q}^{t,s}\bydef(\dot{B}^t_{p,q})_h(\dot{B}^s_{p,q})_v$. If $p=q=2$ then this last space is equivalent to $\dot{H}^{t,s}$. More precisely, we have:
\begin{equation*}
\norm a _{\dot{B}_{2,2}^{t,s}}^2 \approx \norm a_{\dot{H}^{t,s}}^2 \bydef \int_{\mathbb{R}^3}|\xi_h|^{2t} |\xi_v|^{2s} |\hat{a}(\xi)|^2 d\xi.
\end{equation*}
\end{defin}
\section{Technical lemmas} In this part we present seven lemmas used in the previous section, we will prove the three last ones and give references for the four first ones.\\
We start by recalling a Bernstein type lemma from \cite{Chemin,Chemin2}    
\begin{lemma} \label{ber}
Let $\mathcal{B}_h$ (resp. $\mathcal{B}_v$) be a ball of $\mathbb{R}^2_h$ (resp. $\mathbb{R}_v$) and $\mathcal{C}_h$ (resp. $\mathcal{C}_v$) be a ring of $\mathbb{R}^2_h$ (resp. $\mathbb{R}_v$). Let also $a$ be a tempered distribution and $\hat{a}$ its Fourier transform. Then for $1\leq p_2\leq p_1 \leq \infty$ and $1\leq q_2\leq q_1\leq \infty$ we have:
\begin{align*}
&\text{Supp }\hat{a} \subset 2^k\mathcal{B}_h \ \Longrightarrow \ 
\norm {\partial^{\alpha}_{x_h}a}_{L^{p_1}_h(L^{q_1}_v)} \lesssim 2^{k\big( |\alpha| + 2\big( \frac{1}{p_2}- \frac{1}{p_1}\big) \big)} \norm {a}_{L^{p_2}_h(L^{q_1}_v)} \\ 
&\text{Supp }\hat{a} \subset 2^l\mathcal{B}_v \ \Longrightarrow \  
\norm {\partial^{\beta}_{x_3}a}_{L^{p_1}_h(L^{q_1}_v)} \lesssim 2^{l\big( \beta + \big( \frac{1}{q_2}- \frac{1}{q_1}\big) \big)} \norm {a}_{L^{p_1}_h(L^{q_2}_v)} \\
&\text{Supp }\hat{a} \subset 2^k\mathcal{C}_h \ \Longrightarrow \
\norm {a}_{L^{p_1}_h(L^{q_1}_v)} \lesssim 2^{-kN} \sup_{|\alpha|=N}\norm {\partial^{\alpha}_{x_h}a}_{L^{p_1}_h(L^{q_1}_v)} \\
&\text{Supp }\hat{a} \subset 2^l\mathcal{C}_v \ \Longrightarrow \
\norm {a}_{L^{p_1}_h(L^{q_1}_v)} \lesssim 2^{-lN} \norm {\partial^{N}_{x_3}a}_{L^{p_1}_h(L^{q_1}_v)}.
\end{align*}
\end{lemma}
 Let us also recall an anisotropic version of the usual product laws in Besov spaces (see Lemma 4.5 from \cite{Chemin2})
\begin{lemma}\label{produclaw}
Let $q\geq 1$, $p_1\geq p_2\geq 1$ with $\frac{1}{p_1}+ \frac{1}{p_2} \leq 1$, and $s_1< \frac{2}{p_1}$, $s_2< \frac{2}{p_2}$ (resp. $s_1\leq \frac{2}{p_1},$ $s_2\leq \frac{2}{p_2}$ if $q=1$) with $s_1+s_2>0$. Let $\sigma_1<\frac{1}{p_1}$, $\sigma_2 < \frac{1}{p_2}$ (resp. $\sigma_1\leq \frac{1}{p_1}$, $\sigma_2 \leq \frac{1}{p_2}$ if $q=1$) with $\sigma_1+\sigma_2>0$. Then for $a$ in $\dot{B}^{s_1,\sigma_1}_{p_1,q}$ and $b$ in $\dot{B}^{s_2,\sigma_2}_{p_2,q}$, the product $ab$ belongs to $ \dot{B}^{s_1+ s_2 -\frac{2}{p_2},\sigma_1+ \sigma_2 - \frac{1}{p_2}}_{p_1,q}$ and we have
\begin{equation*}
\norm {ab}_{\dot{B}^{s_1+ s_2 -\frac{2}{p_2},\sigma_1+ \sigma_2 - \frac{1}{p_2}}_{p_1,q}} \lesssim \norm a_{\dot{B}^{s_1,\sigma_1}_{p_1,q}} \norm b_{\dot{B}^{s_2,\sigma_2}_{p_2,q}}
\end{equation*}
\end{lemma}
 A very useful lemma in the anisotropic context (lemma 4.3 from \cite{Chemin2}), is the following
\begin{lemma}\label{embedding-anisotropic}
For any $s$ positive, for all $(p,q)\in [1,\infty]$ and any $\theta \in ]0,s[$, we have 
\begin{equation*}
\norm f_{(\dot{B}^{s-\theta}_{p,q})_h(\dot{B}^{\theta}_{p,1})_v} \lesssim  \norm f_{\dot{B}^s_{p,q}}.
\end{equation*}
\end{lemma}
 Finally, we recall lemma A.2 from \cite{Isab}
\begin{lemma}\label{lemma-from-isab}
For any function $a$ in the space $\dot{H}^{\frac{1}{2} + s}(\mathbb{R}^3)$ with $\frac{1}{2} \leq s < 1$, there holds
\begin{equation*}
\norm a_{L^{\infty}_v(\dot{H}^{s}_h)} \leq \sqrt{2} \norm a_{\dot{H}^{\frac{1}{2}+s}(\mathbb{R}^3)}.
\end{equation*}
\end{lemma}
 Next, we will prove an interpolation version in space-time spaces
\begin{lemma}\label{interpolation-lemma}
For all $p\in ]2, \infty[$, there exists a constant $c_p>0$, such that for all $u$ in $L^{\infty}_T(L^2(\mathbb{R}^3)) \cap L^2_T(\dot{H}^1(\mathbb{R}^3))$ we have
\begin{equation*}
\norm {u}_{L^p_T(\dot{B}^{\frac{2}{p}}_{2,1}(\mathbb{R}^3))} \leq c_p \norm u_{L^2_T(\dot{H}^1(\mathbb{R}^3))}^{\frac{2}{p}} \norm u_{L^{\infty}_T(L^2(\mathbb{R}^3))}^{1-\frac{2}{p}}.
\end{equation*}
\end{lemma}
\begin{proof}[\textit{Proof}]
The proof is classical, we proceed as the following:\\
Let $N(t)>0$ to be fixed later, we use lemma \ref{ber} and Cauchy-Swartz inequality, to write
\begin{align*}
\sum_{j\in \mathbb{Z}} 2^{\frac{2}{p}} \norm {\Delta_j u(t,.)}_{L^2} &= \sum_{j\leq N(t)} 2^{j\frac{2}{p}} \norm {\Delta_j u(t,.)}_{L^2} +\sum_{j>N(t)} 2^{j(\frac{2}{p}-1)} \norm {\Delta_j \nabla u(t,.)}_{L^2} \\
&\leq 2^{\frac{2}{p}N(t)} \norm {u(t,.)}_{L^2} + 2^{(\frac{2}{p}-1)N(t)} \norm {\nabla u(t,.)}_{L^2}.
\end{align*} 
The choice of $N(t)$ such that
\begin{equation*}
2^{N(t)} \bydef \bigg( \frac{1-\frac{2}{p}}{\frac{2}{p}} \bigg) \frac{\norm {\nabla u(t,.)}_{L^2}}{\norm {u(t,.)}_{L^2}}
\end{equation*}
gives
\begin{equation}\label{ineterpolation-inequa}
\norm {u(t,.)}_{\dot{B}^{\frac{2}{p}}_{2,1}(\mathbb{R}^3))} \leq c_p \norm {u(t,.)}_{\dot{H}^1(\mathbb{R}^3)}^{\frac{2}{p}} \norm {u(t,.)}_{L^2(\mathbb{R}^3)}^{1-\frac{2}{p}}.
\end{equation}
The lemma follows by taking the $L^p$ norm in time.
\end{proof}
 The following lemma can be used when the horizontal frequencies control the vertical ones
\begin{lemma}\label{goodfrequencieslemma}
Let $s,t$ be two real numbers, let $f$ be a regular function, we define $f^{\textbf{G}}$ as
$$ f^{\textbf{G}} \overset{\text{def}}{=} \sum_{q\leq k} \Delta_k^h \Delta_q^v  f. $$
Then we have:
\begin{equation*}
\norm {\partial_3 f^{\textbf{G}}}_{\dot{H}^{s,t}} \lesssim \norm {\nabla_h f}_{\dot{H}^{s,t}}.
\end{equation*}
\end{lemma}
\begin{proof}[\textit{Proof}]
Let us use Plancherel-Parseval identity to write:
\begin{equation}\label{A1}
\norm {\partial_3 f^{\textbf{G}}}_{H^{s,t}} ^2 \approx \int_{\mathbb{R}^3} |\xi_h|^{2s} |\xi_v|^{2t} \bigg|\sum_{q \leq k} |\xi_v|\varphi_k^h(\xi)\varphi_q^v(\xi)\bigg|^2 |\hat{f}(\xi)|^2 d\xi,
\end{equation}
where: $\varphi_k^h \bydef \varphi (2^{-k} |\xi_h|) $, $\varphi_q^v \bydef \varphi (2^{-q} |\xi_v|) $, and $\varphi$ is the function defined at the beginning of the Appendix part.
Thus, using the support properties of $\varphi_k^h$, $\varphi_q^v$, and the condition $q<k$, we infer that, for all $\xi=(\xi_h,\xi_v)\in Supp(\varphi_k^h) \times supp(\varphi_q^v)$
\begin{equation}\label{A2}
|\xi_v|\lesssim 2^{q} \leq 2^{k} \lesssim |\xi_h|
\end{equation}
plugging \eqref{A2} into \eqref{A1} concludes the proof of the lemma.
\end{proof}
 The last lemma that we will prove is useful to estimate some parts of the anisotropic Bony's decomposition for functions having dominated vertical frequencies compared to the horizontal ones, and which are supported away from zero horizontally in Fourier side.
\begin{lemma}\label{badfrequencieslemma}
Let $f$ be regular function, and $E>0$. We define $f^{\textbf{B}}_{\sharp,E^{-1}}$ as
$$ f^{\textbf{B}}_{\sharp,E^{-1}} \overset{\text{def}}{=} \sum_{k< q} \Delta_k^h \Delta_q^v f_ {\sharp,E^{-1}},$$
where
$$ f_ {\sharp,E^{-1}}\overset{\text{def}}{=}  \mathcal{F}^{-1} \big(\mathbf{1}_{B_h^c(0,E^{-1}} )\hat{f} \big).$$
Then we have the following estimates
\begin{equation}\label{bad1}
\norm {\Delta_q^v S_{j-1}^h (f^{\textbf{B}}_{\sharp,E^{-1}})}_{L^2_v(L^{\infty}_h)} \lesssim \big( log(E2^q+e)\big)^{\frac{1}{2}} c_q \norm {\nabla_h f}_{L^2(\mathbb{R}^3)},
\end{equation}
\begin{equation}\label{bad2}
\norm {S_{q-1}^v S_{j-1}^h (f^{\textbf{B}}_{\sharp,E^{-1}})}_{L^{\infty}_v(L^{\infty}_h)} \lesssim \big( log(E2^q+e)\big)^{\frac{1}{2}} c_q 2^{\frac{q}{2}} \norm {\nabla_h f}_{L^2(\mathbb{R}^3)}.
\end{equation}
\end{lemma}
\begin{proof}[\textit{Proof}]
According to the support properties we have
\begin{align*}
\Delta_q^v S_{j-1}^h (f^{\textbf{B}}_{\sharp,E^{-1}})= \bigg(\Delta_q^v S_{j-1}^h \sum_{i\in\{-1,0,1 \}} S_{q-1+i}^h \Delta_{q+i}^v \bigg)f_ {\sharp,E^{-1}},
\end{align*}
therefore, Bernstein's inequality, we can write
\begin{align*}
\norm {\Delta_q^v S_{j-1}^h (f^{\textbf{B}}_{\sharp,E^{-1}})}_{L^2_v(L^{\infty}_h)} &\lesssim \sum_{E^{-1}\lesssim 2^k \lesssim 2^q} 2^k \norm {\Delta_k^h \Delta_q^v f}_{L^2(\mathbb{R}^3)}\\
&\lesssim \bigg(\sum_{E^{-1}\lesssim 2^k \lesssim 2^q} c_k \bigg)c_q \norm {\nabla_h f}_{L^2(\mathbb{R}^3)}\\
&\lesssim \big( log(E2^q+e)\big)^{\frac{1}{2}} c_q  \norm {\nabla_h f}_{L^2(\mathbb{R}^3)}.
\end{align*}
Thus the first inequality is proved. For the second one, we first write 
\begin{align*}
\norm {S_{q-1}^v S_{j-1}^h (f^{\textbf{B}}_{\sharp,E^{-1}})}_{L^{\infty}_v(L^{\infty}_h)} \lesssim \sum_{m\leq q} 2^{\frac{m}{2}}\norm {\Delta_m^v S_{j-1}^h (f^{\textbf{B}}_{\sharp,E^{-1}})}_{L^{2}_v(L^{\infty}_h)}.
\end{align*}
Inequality \eqref{bad1} gives then
\begin{align*}
2^{-\frac{q}{2}}\norm {S_{q-1}^v S_{j-1}^h (f^{\textbf{B}}_{\sharp,E^{-1}})}_{L^{\infty}_v(L^{\infty}_h)} &\lesssim \bigg( \sum_{m\leq q} 2^{\frac{1}{2} (m-q)} \big( log(E2^m+e)\big)^{\frac{1}{2}} c_m  \bigg)\norm {\nabla_h f}_{L^2(\mathbb{R}^3)}\\
&\lesssim \big( log(E2^q+e)\big)^{\frac{1}{2}} c_q  \norm {\nabla_h f}_{L^2(\mathbb{R}^3)}.
\end{align*}
Inequality \eqref{bad2} follows.
\end{proof}
\section*{Acknowledgment} 
 The author is very grateful to the referee for his/her valuable and helpful comments and remarks.

\bibliographystyle{elsarticle-num}
 
 \bibliography{mybibfile.bib}
 
\end{document}